\documentclass[11pt]{article}

\usepackage{amsfonts, amsbsy, amsmath, amssymb, amsthm}

\usepackage{graphicx,wrapfig,url}

\newtheorem{theorem}{Theorem}[section]
\newtheorem{corollary}[theorem]{Corollary}
\newtheorem{lemma}[theorem]{Lemma}
\newtheorem{proposition}[theorem]{Proposition}

\newtheorem{definition}[theorem]{Definition}

\newtheorem{remark}[theorem]{Remark}

\newtheorem{conjecture}[theorem]{Conjecture}

\input{epsf.sty}
\usepackage{graphicx}

\begin{document}

\title{Quandle Identities and Homology}

\author{W. Edwin Clark \ and \ Masahico  Saito \\
Department of Mathematics and Statistics\\ University of South Florida
 }

\date{\empty}

\maketitle

\begin{abstract}
Quandle homology was defined from rack homology as the quotient by 
a subcomplex corresponding to the idempotency, for invariance under the type I Reidemeister move.
Similar subcomplexes have been considered for various identities of racks and moves on diagrams.
We observe common aspects of these identities and subcomplexes; a quandle identity gives rise to a $2$-cycle, 
 the abelian extension with a $2$-cocycle that vanishes on the $2$-cycle 
inherits the identity, 
and  a subcomplex is constructed from the identity.
Specific identities are examined among small connected quandles.
\end{abstract}

\section{Introduction}

Quandle homology \cite{CJKLS} was defined from rack homology \cite{FRS1}
as the quotient by 
a subcomplex corresponding to the idempotency, for invariance under the type I Reidemeister move.
Similar subcomplexes have been considered for various identities of racks and moves on diagrams.
Typically  a certain change of  knot diagrams requires a condition for  quandle $2$-cocycles to satisfy
to obtain desired cocycle invariant, and the condition leads to a subcomplex.
For example, for defining quandle cocycle invariants for unoriented knots, a good involution was defined in \cite{KO},
and a corresponding condition for cocycles and a subcomplex were defined. For a telephone cord move for racks, 
a condition for rack $2$-cocycles and a subcomplex were defined in \cite{EN}. 
Certain moves on handlebody-links were considered in \cite{CIST,IIJO}, and corresponding subcomplexes were defined.
In this paper, we give a construction of a $2$-cycle  $L$ from a given  identity  of quandles in a certain form, 
show that the abelian extension  with a $2$-cocycle $\phi$ such that $\phi(L)=0$ 
inherits the identity,  
and construct  a  subcomplex from the identity.
The homology and cohomology of these subcomplexes remain to be investigated.

Preliminary material and definitions are provided in Section~\ref{sec:prelim}, and
an outline of the method is explained using type 3 racks in Section~\ref{sec:type3}.
The main results are presented in Section~\ref{sec:main} with proofs, and applied to 
type $n$ quandles and identities similar to the Burnside relations. 
In Section~\ref{sec:id}, specific identities  are examined among small connected quandles,
and a large class of Alexander quandles is given that satisfies many identities.

\section{Preliminary} \label{sec:prelim}

In this section, we provide preliminary material, definitions and notation.
More details can be found, for example, in \cite{CKS,FR}. 

A {\it rack} $X$ is a non-empty set with a binary operation $(a, b) \mapsto a * b$
satisfying the following conditions.

\medskip

%\begin{eqnarray*}
 {\rm (1)} \   For any $b \in X$, the map $R_b: X \rightarrow X$ defined by $R_b(a)=a*b$ for
$a \in X$ is a bijection.
%} \label{axiom2} \\

  {\rm (2)} \ 
 For any $a,b,c \in X$, we have
$ (a*b)*c=(a*c)*(b*c). $
%} \label{axiom3} 
%\end{eqnarray*}

\medskip

The map $R_b$ in the first  axiom is called the {\it right translation by $b$}. By axioms $R_b$ is a {\it  rack isomorphism}.
A {\it quandle } $X$  is a rack  with  idempotency: $a*a=a$ for any $a \in X$.
 A {\it quandle homomorphism} between two quandles $X, Y$ is
 a map $f: X \rightarrow Y$ such that $f(x*_X y)=f(x) *_Y f(y) $, where
 $*_X$ and $*_Y$ 
 denote 
 the quandle operations of $X$ and $Y$, respectively.
A {\it generalized Alexander quandle} is defined  by  
a pair $(G, f)$ where 
$G$ is a  group,  $f \in {\rm Aut}(G)$,
and the quandle operation is defined by 
$x*y=f(xy^{-1}) y $.  
If $G$ is abelian, this is called an {\it Alexander quandle}.

Let $X$ be a rack.
For brevity we sometimes omit $*$ and parentheses, 
so that for $x_i \in X$, $x_1x_2=x_1 * x_2$, 
$x_1 x_2 x_3 = (x_1 x_2)x_3$, and inductively, 
$x_1 \cdots x_{k-1} x_k=(x_1 \cdots x_{k-1} ) x_k$. 

We  also use the notation $x*^n y = x*y *\cdots * y $ where $y$ is repeated $n$ times.
A rack $X$ is  said to be of  {\it type $n$} (cf.~\cite{Joyce2})
 if $n$ is the least  positive integer such that  $x *^n y=x$ holds for all $x, y \in X$,
and we write ${\rm type}(X)=n$. A type 1 quandle is called {\it trivial}, and a 
type 2 quandle is called a {\it kei} or an {\it involutory } quandle. 

The subgroup of ${\rm Sym}(X)$ generated by the permutations ${ R}_a$, $a \in X$, is 
called the {\it inner automorphism group} of $X$,  and is 
denoted by ${\rm Inn}(X)$. 
A rack is {\it connected} if ${\rm Inn}(X)$ acts transitively on $X$.

The rack chain group $C_n(X)=C^R_n(X)$ for a rack $X$ is defined to be the free abelian group generated by 
$n$-tuples $(x_1, \ldots, x_n)$, $x_i \in X$ for $i=1, \ldots, n$.
Let $d_h^{(n)}, \delta _h^{(n)} : C_n(X) \rightarrow C_{n-1}$ be defined by 
\begin{eqnarray*}
d_h ^{(n)} (x_1, \ldots, x_h, \ldots, x_n) &=&   (x_1, \ldots, \widehat{x_h},  \ldots, x_n),   \\
\delta _h^{(n)}  (x_1, \ldots, x_h, \ldots, x_n) &=&  (x_1 * x_h , \ldots, x_{h-1} * x_h , \widehat{x_h}, \ldots, x_n), 
\end{eqnarray*}
respectively, where $\hat{ \ } $ denotes deleting the entry.
Then the boundary map is defined by $\partial_n=\sum_{h=2}^{n} (-1)^h [ d_h^{(n)} - \delta _h^{(n)} ] $.
The degeneracy subcomplex $C^D(X)$ was defined   \cite{CJKLS}  for a quandle $X$ with generating terms
$(x_i)_{i=1}^n \in C_n(X) $ with $x_j=x_{j+1}$ for some $j=1, \ldots, n-1$,
and the quotient complex $\{ C^Q_n(X)=C^R_n(X) / C^D_n(X), \partial_n \}$ was defined \cite{CJKLS}
as the quandle homology.

The corresponding $2$-cocycle is formulated as follows.
  A  quandle $2$-cocycle  is regarded as  a function $\phi: X \times X \rightarrow A$ for an abelian group $A$
  that satisfies $$ \phi (x, y)-
    \phi(x,z)+ \phi(x*y, z) - \phi(x*z, y*z)=0$$ 
    for any $x,y,z \in X$ and 
    $\phi(x,x)=0$ for any $x\in X$. For a quandle $2$-cocycle $\phi$, $E=X \times A$ becomes a quandle by
    \[
    (x, a) * (y, b)=(x*y, a+\phi(x,y))
    \]
    for $x, y \in X$, $a,b \in A$, denoted by
    $E(X, A, \phi)$ or simply $E(X, A)$, and it is called an \emph{abelian  extension} of $X$ by $A$. 
    The second factor, in this case, is written in additive notation of $A$. See \cite{CENS,CKS} for more details.

Computations using \textsf{GAP}~\cite{Leandro} significantly expanded the list
of small connected quandles.
These
quandles, called {\it Rig} quandles,
may be found in the \textsf{GAP}~package Rig \cite{rig}.  
Rig includes all  connected quandles of order less than 48, at this time. 
Properties of some of Rig quandles,  such as homology groups and cocycle invariants,
are also found in  \cite{rig}.
We  use the notation $Q(n,i)$
for the $i$-th quandle of order $n$ in the list of Rig quandles, denoted in \cite{rig} by
{\sf SmallQuandle}$(n,i)$.
Note, however, that in \cite{rig} quandles are left distributive, so that as a matrix,  $Q(n,i)$
is the transpose of the quandle matrix {\sf SmallQuandle}$(n,i)$ in
\cite{rig}.

\section{Type $3$ quandles}\label{sec:type3}

Before presenting the main theorem and proof,  we describe  the properties of rack identities 
through the example of type $3$ quandles in this section.
We note that a subcomplex for type 2 quandles, or keis, is defined in \cite{KO}
as a special case of their subcomplex.
Recall that a rack $X$ is of type 3 if it satisfies  the identity $S$: $x*y*y*y=x$ for any $x, y \in X$. 
We observe the following three properties.

(i)  From this identity $S$ we form  a $2$-chain 
$$L=L_S=(x, y)+(x * y, y) + (x*y*y, y).$$
 It is checked that $L$ is a $2$-cycle:
$$\partial (L)=[\ (x) - (x  *  y) \ ] + [ \ (x  *  y) - (x  *  y  *  y ) \ ] + [ \ (x  *  y*y ) - (x  *  y  *  y *y) \ ] = 0, $$
using the identity $S$.

(ii)
Let $\phi \in Z^2_R(X, A)$ be a rack $2$-cocycle with the coefficient abelian group $A$ such that $\phi(L)=0$. 
Then for $E(X,A,\phi)=X \times A$, one computes
\begin{eqnarray*}
\lefteqn{(x, a)*(y,b)*(y,b)*(y,b)}\\
&=& 
(x*y, a+\phi(x,y) )*(y,b)*(y,b) \\
&=& (x*y*y*y, a+\phi(x,y) +\phi(x*y,y)+  \phi(x*y*y,y) ) \\
&=& (x,a). 
\end{eqnarray*}
Hence $E(X,A,\phi)$ is of type 3.

(iii)
Define, for each $n$, a subgroup $C^S_n(X) \subset C_n(X)$ generated by 
\begin{eqnarray*}
\bigcup_{j=1}^{n-1} \{ \  & & (x_1, \ldots, x_j, y, x_{j+2}, 
\ldots, x_n) \\
&+& (x_1 * y  , \ldots, x_j*y, y, x_{j+2} , \ldots, x_n) \\
&+&  (x_1 * y *y  , \ldots, x_j*y * y , y, x_{j+2} , \ldots, x_n) \  \\
& & | \ x_i, y \in X, \, i=1, \ldots, \widehat{j+1}, \ldots, n \ \} . 
\end{eqnarray*}
For a fixed $j$, $y$ is positioned at $(j+1)$-th entry. 
Then $\{ C_n^S, \partial_n \}$ is a subcomplex, which will be proved in the general case in Section~\ref{sec:main}.
In this section, to illustrate the idea of the proof, we compute the image under the boundary map for 
 a specific generator of $C^S_4(X)$.
We simplify  the notation and use $(1,2,3,4)$ for 
$(x_1, x_2, x_3, x_4)$, $12$ for $x_1*x_2$, etc. 
Let $c=(1,2,3,4) + (13, 23, 3, 4) + (133, 233, 3, 4)$ that is a generator for  the case $n=4$ and $j=2$.
We compute $\partial_4 (c)$ and show that the image is in $C_3^S(X)$. 
First we compute 
$$d_2^{(4)} (c) =  (1,3,4)+(13,3,4)+(133,3,4) 
$$
which is a generator of  $C_3^S(X)$.
Then we have 
$$
\delta_2^{(4)}(c) = (12,3,4) + ((13)(23), 3,4) + ((133)(233), 3,4) .
$$
One computes that $(13)(23)=(12)3$, and 
$$(133)(233)=[ ( 13) 3 ] [ ( 23 ) 3 ] = [ ( 13) (23) ] 3 = [ (12) 3 ] 3, $$
so that 
$\delta_2^{(4)}(c)=  (12, 3, 4) + (123,3,4) + (1233,3,4), $
which is   a generator of $C_3^S(X)$.
We also compute 
\begin{eqnarray*}
[ d_3^{(4)} -  \delta_3^{(4)}] (c) &=&
[ (1,2,4) - (13,23,4)] + 
[(13,23,4)-(133,233,4) ] \\
& & \quad +\  [ (133,233,4)- (1,2,4)]\ =\ 0,
\end{eqnarray*}
in this case.
Next  we have
$d_4^{(4)} (c)=(1,2,3)+(13,23,3) + (133, 233, 3)$ which is a generatior.
Finally we have 
\begin{eqnarray*}
 \delta_4^{(4)} (c) &=&  (14, 24, 34) + (134, 234, 34) + (1334, 2334, 34) \\
 &=& (14,24,34) + ((14)(34), (24)(34), 34) + ((14)(34)(34), (24)(34)(34), 34) 
 \end{eqnarray*}
 which is a generator.
This concludes the computation that  $\partial_4 (c) \in C_3^S(X)$.

\section{From identities to extensions and subcomplexes}\label{sec:main}

Let $X$ be a rack.
For brevity we omit $*$ and take the left-most parenthesis as before. 
Fix a surjection  $\tau: \{ 1, \ldots, k \} \rightarrow \{ 1, \ldots, m \}$, where $k \geq m$ are positive integers.
We consider identities of the form 
$x y_{\tau(1)} \cdots y_{\tau(k)} =x$ for $x, y_{\tau(i)} \in X$, $i=1, \ldots, m$.
The expression $y_{\tau(1)} \cdots y_{\tau(k)}$ is a word of length $k$ from the alphabet 
$\{ y_1, \ldots, y_m \}$. 
We assume that $k>1$, since otherwise the quandle is trivial.

For example, for a type $n$ rack $X$, there is an identity  of the form
$ x \, \underbrace{y_1 \cdots y_1}_{k} = x $ for any $x, y_1 \in X$. 
Another example is $ x  \underbrace{y_1 y_2\cdots y_1 y_2}_{2k} = x $.

\begin{definition}
{\rm 

We call 
an identity $S$ of the form $x y_{\tau(1)} \cdots y_{\tau(k)} =x$ as described above
a $(\tau, k, m)$ 
{\it  inner identity}.

If an inner identity $S$ above holds  for any $x, y_j \in X$, $j=1, \ldots, m$,
then we say that $X$ satisfies the $(\tau, k, m)$ 
inner identity $S$.
}
\end{definition}

A rack $X$ satisfies a $(\tau, k,m)$ inner identity $S$ 
 of the form $x y_{\tau(1)} \cdots y_{\tau (k)} =x$ if and only if 
$  R_{y_{\tau(k)}} \cdots R_{y_{\tau(2)}} R_{y_{\tau(1)}} = {\rm id}  \in  {\rm Inn}(X)$
for all $y_j \in X$, $j=1, \ldots, m$.

\begin{definition}
{\rm

Let $S$ be an inner identity  $x  y_{\tau(1)} \cdots y_{\tau (k)} =x$.
Set $ \omega_i = y_{\tau(1)} \cdots y_{\tau(i)} $, when $i > 0$. 
Define a $2$-chain $L_S$ by 
\begin{eqnarray*}
L_S & = &  (x, y_{\tau(1)}) +  \sum_{i=1}^{k-1} ( x y_{\tau(1)}  \cdots   y_{\tau(i)}, y_{\tau(i+1)} )  \\
&=& (x, \omega_1) +  \sum_{i=1}^{k-1} ( x \omega_i , y_{\tau(i+1)} ) \\
&=&  (x, \omega_1) + (x \omega_1, y_{\tau(2)}) + \cdots + (x  \omega_{k-1},  y_{\tau (k)} )  .
\end{eqnarray*}
}
\end{definition}

\begin{definition}
{\rm 
Let $S$: $x  y_{\tau(1)}\cdots y_{\tau(k)}=x$ be an inner identity.
Set $ \omega_i = y_{\tau(1)} \cdots y_{\tau(i)} $, when $i > 0$. 
 Let $C^S_n (X) \subset C_n(X)$, $n \in \mathbb{Z}$, be subgroups generated by
\begin{eqnarray*}
 \lefteqn{ 
   \bigcup_{j=1}^{k-1} \ 
 \{ \  
    (x_1, \ldots, x_j, y_{\tau(1)}, x_{j+2}, \ldots, x_n)  
    }\\
  & &   + \sum_{i=1}^{k-1} (x_1 \omega_i , \ldots, x_j \omega_i, y_{\tau(i+1)}, x_{j+2}, \ldots, x_n )
  \\
& & \ | \ 
x_h, y_{\tau(i)} \in X,\,  h=1, \ldots, \widehat{j+1}, \ldots, n, \,  i=1, \ldots, k-1
\ \} . 
\end{eqnarray*}
}
\end{definition}

We use the following lemma in the proof of Theorem~\ref{thm:main}.

\begin{lemma} \label{lem:reduce}
Let $X$ be a quandle.
\begin{itemize}
\item[{\rm (i)}]
For any $a, b, c_i \in X$,  it holds that 
$(a c_1 \cdots c_i )( b c_1 \cdots c_i ) = (ab) c_1 \cdots c_i $.
\item[{\rm (ii)}]
For any $a_i,  b \in X$,  it holds that 
$a_1 \cdots a_i b = (a_1 b) \cdots (a_i b)$. 
\end{itemize}

\end{lemma}

\begin{proof}
(i) One computes inductively, using self-distributivity,  
\begin{eqnarray*}
(a c_1 \cdots c_i )( b c_1 \cdots c_i ) &=& [ ( a c_1 \cdots c_{i-1} ) c_i ] [ ( b c_1 \cdots c_{i-1} ) c_i ] \\
 & = & [ ( a c_1 \cdots c_{i-1} )  ( b c_1 \cdots c_{i-1} ) ] c_i  \\
 & = & [ [  ( a c_1 \cdots c_{i-2} )  ( b c_1 \cdots c_{i-2} ) ] c_{i-1}  ]  c_i  \\
 & =& \cdots \\
 &=&  (ab) c_1 \cdots c_i .
\end{eqnarray*}
(ii) One computes inductively 
\begin{eqnarray*}
a_1 \cdots a_i b 
&=& 
(a_1 \cdots a_{i-1} b ) ( a_i b ) \\
& = & 
[ (a_1 \cdots a_{i-2} b ) (a_{i-1} b) ]  ( a_i b ) \\
& =& \cdots \\
 &=&
(a_1 b) \cdots (a_i b) 
 \end{eqnarray*}
as desired.
\end{proof}

\begin{figure}[htb]
    \begin{center}
   \includegraphics[width=4.5in]{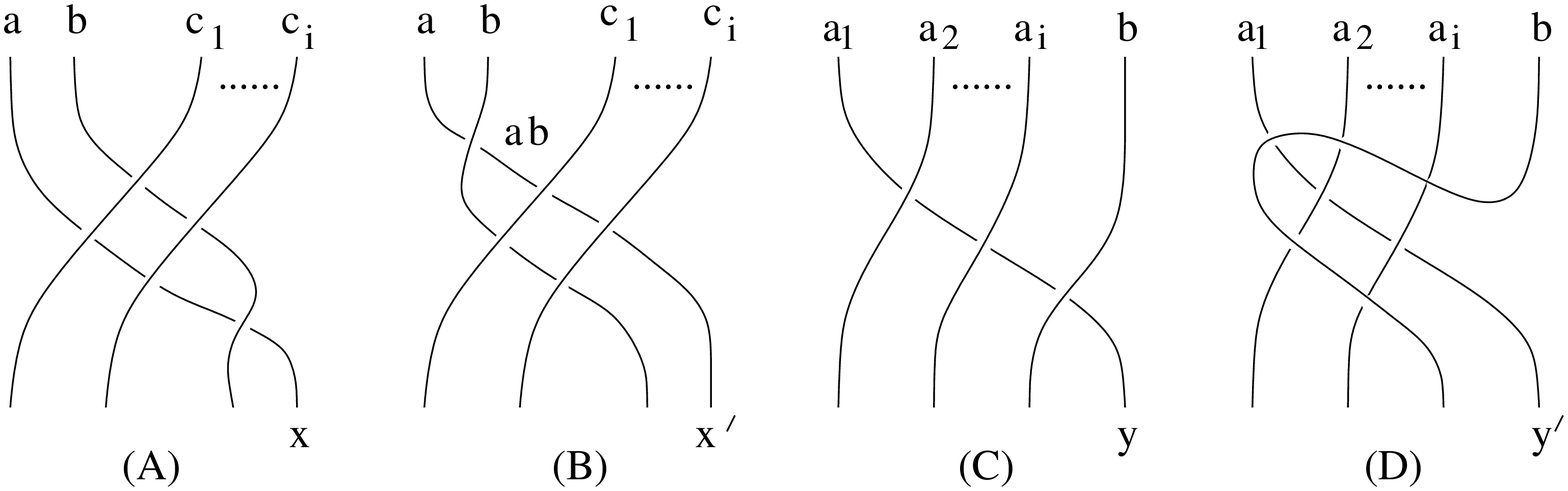}\\
    \caption{ Diagrams for Lemma~\ref{lem:reduce} }\label{lemma}
    \end{center}
\end{figure}

\begin{remark}
{\rm

The proof of Lemma~\ref{lem:reduce} has an diagrammatic interpretation as depicted in Figure~\ref{lemma}. 
Colorings of knot diagrams by quandles are well known and extensively used to construct knot invariants.
At a crossing, for  positive crossings in the figure with all arcs oriented downwards, 
the coloring condition is as depicted at the top left crossing
of Figure~\ref{lemma} (B), where $a$ and $b$ are assigned on the left top under-arc and the over-arc, respectively,
and $a*b$ (simply denoted by $ab$) is required to be assigned on the other under-arc.
Then the bottom right arc of (A) receives $x=(a c_1 \cdots c_i )( b c_1 \cdots c_i ) $.
In (B) the corresponding arc receives $x'=(ab) c_1 \cdots c_i $. 
Thus the fact that the colorings are in bijection under Reidemeister moves shows the equality (i). 
Similarly, (C) and (D) represents the equality (ii), with 
$y=a_1 \cdots a_i b $ and $y'=(a_1 b) \cdots (a_i b) $.
One also sees that the inductive calculations in the proof can be represented by step by step moves.
}
\end{remark}

\begin{theorem}\label{thm:main}
\begin{sloppypar}
Let $X$ be a rack.
Let $S$ be a $(\tau, k, m)$ inner identity  $x  y_{\tau(1)}   \cdots   y_{\tau(k)}=x $ that $X$ satisfies.
Then the following holds.
\end{sloppypar}

\begin{itemize}
\item[{\rm (i)}] The $2$-chain $L_S$ is a $2$-cycle, $L_S \in Z_2(X)$.

\item[{\rm (ii)}] For an abelian group $A$ and a $2$-cocycle $\phi$,  $E(X,A, \phi)$ satisfies $S$ if and only if $\phi(L_S)=0$. 

\begin{sloppypar}
\item[{\rm (iii)}] The sequence of subgroups $C^S_n (X) \subset C_n(X)$ form a subcomplex $\{ C^S_n (X), \partial_n \}$,  $n \in \mathbb{Z}$.
\end{sloppypar}

\end{itemize}
\end{theorem}

\begin{proof}
(i) Set $ \omega_i = y_{\tau(1)} \cdots y_{\tau(i)} $ when $i > 0$ as before.
One computes 
\begin{eqnarray*}
\partial  \left( (x, \omega_1) +  \sum_{i=1}^{k-1} ( x \omega_i , y_{\tau(i+1)} ) \right) 
 &=& (x) + \sum_{i=1}^{k-1}\  [ \ ( x   \omega_i )  - (  x   \omega_{i+1}) \ ] \\
 &=& (x) - (x  \omega_k  )  \  = \ 0 
\end{eqnarray*}
as desired.

\noindent 
(ii) For  $(y_{\tau(i)}, a_i) \in E=X \times A$, $i=0, \ldots, k$, one computes, inductively, 
\begin{eqnarray*}
\lefteqn{   (x, a_0)      \cdots     (y_{\tau(k)} , a_k) }  \\
 &=&   (x     y_{\tau(1)}, a_0 + \phi (x, y_{\tau(1)}) )      \cdots     (y_{\tau(k)} , a_k) \\
  &=&  (x     y_{\tau(1)}     y_{\tau(2)}, a_0 + \phi (x, y_{\tau(1)}) + \phi (x     y_{\tau(1)}, y_{\tau(2)})  )     \cdots     (y_{\tau(k)} , a_k)\\
  &=&  \cdots \\
 &=&  (\ x    \omega_k , a_0 + ( \ \phi (x, \omega_1)   + \sum_{i=0}^{k-1}  \phi (x \omega_i , y_{\tau(i+1)})  \ ) \ ) ,
\end{eqnarray*}
which is equal to $(x, a_0)$ if and only if $\phi(L_S)=0$, as desired. 

\noindent
(iii) We check the following three cases.

Case (1):  $h \leq  j $.
In this case, each term of 
\begin{eqnarray*}
\lefteqn{ 
d_h ^{(n)} ( \  (x_1, \ldots, x_j, y_{\tau(1)}, x_{j+2}, \ldots, x_n) 
} \\
& & 
 +  \sum_{i=1}^{k-1}
( x_1  \omega_i , \, 
\ldots, 
x_{j-1} \omega_{i}, \,  
x_j  \omega_i, \  
y_{\tau(i+1)}, \, x_{j+2},  \ldots, x_n) \ )
\end{eqnarray*}
is obtained from the original term by deleting the $h$-th entry, 
hence the image is an element of $C^S_{n-1} (X)$.
Each term of 
\begin{eqnarray*}
\lefteqn{ 
\delta_h ^{(n)} ( \  (x_1, \ldots, x_j, y_{\tau(1)}, x_{j+2}, \ldots, x_n) 
} \\
& & 
+  \ \sum_{i=1}^{k-1}
( x_1 \omega_i  ,\,  
\ldots, x_{j-1}  \omega_i , \,   
x_j   \omega_i , \, 
 y_{\tau(i+1)}, \, x_{j+2},  \ldots, x_n) \ )
\end{eqnarray*}
is obtained from the original by replacing the first $h$ entries of the form $ x_\ell \omega_i$ 
by 
$$ ( x_\ell \omega_i ) (x_\ell \omega_i) =
( x_\ell   y_{\tau(1)}   \cdots    y_{\tau(i)} ) (  x_h  y_{\tau(1)}   \cdots    y_{\tau(i)} ) . $$
By Lemma~\ref{lem:reduce} (i), we obtain 
$$( x_\ell   y_{\tau(1)}   \cdots    y_{\tau(i)} ) (  x_h  y_{\tau(1)}   \cdots    y_{\tau(i)} ) = (x_\ell  x_h )   y_{\tau(1)}   \cdots    y_{\tau(i)} .$$
Hence the image is in $C^S_{n-1} (X)$.

Case (2): $h=j+1$. 
One computes   
\begin{eqnarray*}
\lefteqn{
d_{j+1} ^{(n)} ( x_1 \omega_i, \,   
\ldots, x_{j}   \omega_i, \, 
y_{\tau(i+1)}, \, x_{j+2},  \ldots, x_n) }\\
&=&
( x_1 \omega_i, \,  
\ldots, x_{j}   \omega_i, \, 
\widehat{y_{\tau(i+1)}},  \, x_{j+2},  \ldots, x_n) 
\end{eqnarray*} 
for the $i$-the term, and 
\begin{eqnarray*}
\lefteqn{
 \delta_{j+1} ^{(n)} ( x_1   \omega_{i-1}, \, 
  \ldots, x_{j-1}  \omega_{i-1}, \, 
   x_j  \omega_{i-1}, \, 
    y_{\tau(i)}, \,  x_{j+2},  \ldots, x_n) ) 
 }\\
&=&
( x_1   \omega_{i}, \, 
  \ldots, x_{j-1}  \omega_{i}, \, 
   x_j  \omega_{i}, \, 
   \widehat{y_{\tau(i)}},    \,  x_{j+2},  \ldots, x_n)
\end{eqnarray*} 
for the $(i-1)$-th term, so that these terms cancel in pairs by opposite signs. 
The first term before the sum over $i$ and the last term of the sum over $i$, 
$$d_{j+1} ^{(n)} ( x_1  , \ldots, x_{j-1}  , x_j , y_{\tau(1)}, x_{j+2},  \ldots, x_n) 
=( x_1  , \ldots, x_{j-1}  , x_j ,  x_{j+2},  \ldots, x_n) 
$$ 
and 
\begin{eqnarray*}
\lefteqn{
\delta_{j+1} ^{(n)} 
( x_1 \omega_{k-1}, \, 
\ldots, x_{j-1}  \omega_{k-1}, \,  
x_j  \omega_{k-1} , \, 
 y_{\tau(k)}, \, x_{j+2},  \ldots, x_n)
} \\
&=& 
( x_1  \omega_{k}, \, 
\ldots, x_{j-1}  \omega_{k}, \,  
x_j  \omega_{k} , \,
\widehat{y_{\tau(k)}}, x_{j+2},  \ldots, x_n)
\end{eqnarray*}
are equal and cancel by opposite signs.
Hence the image of this case is zero.

Case (3): $h>j+1$.
Each term of 
\begin{eqnarray*}
\lefteqn{
d_h ^{(n)} (\  (x_1, \ldots, x_j, y_{\tau(1)}, x_{j+2}, \ldots, x_n) 
} \\
& & 
 + \sum_{i=1}^{k-1}
( x_1  \omega_i, \, 
\ldots, x_{j-1} \omega_i, \,  
x_j   \omega_i, \,  
y_{\tau(i+1)},\,  x_{j+2},  \ldots, x_n) \ )
\end{eqnarray*}
is obtained from the original term by deleting the $h$-th entry, 
hence the image is an element of $C^S_{n-1} (X)$.
Each term of 
\begin{eqnarray*}
\lefteqn{
 \delta_h ^{(n)} (\  (x_1, \ldots, x_j, y_{\tau(1)}, x_{j+2}, \ldots, x_n) 
 }
\\
& & 
  + \sum_{i=1}^{k-1}
( x_1   \omega_i, \,  
\ldots, x_{j-1}   \omega_i, \, 
x_j    \omega_i, \,  
y_{\tau(i+1)},  x_{j+2},  \ldots, x_n) \ )
 \quad (*)
\end{eqnarray*}
is computed as 
$$
( x_1   \omega_i,   
 x_h , \, \ldots,  x_j   \omega_i, 
  x_h , \, y_{\tau(i+1)}  x_h, \, x_{j+2}  x_h,  \ldots , x_{h-1}  x_h, \widehat{x_h}, x_{h+1}, \ldots,  x_n)   .
$$
By Lemma~\ref{lem:reduce} (ii), we obtain
$$ x_\ell \omega_i x_h =  x_\ell  y_{\tau(1)}   \cdots    y_{\tau(i)}   x_h = 
(x_\ell   x_h)   (y_{\tau(1)}   x_h)   \cdots    (y_{\tau(i)}   x_h) .
$$
We note that it holds that if  $y_{\tau(u)}  = y_{\tau(v)}$ then  $y_{\tau(u)}   x_h = y_{\tau(v)}    x_h$.
Hence we can set $y'_{\tau(j)}=y_{\tau(j)} x_h \in X$ for $j=1, \ldots, k$, and $x'_\ell=x_\ell x_h$,  then 
$$(x_\ell   x_h)   (y_{\tau(1)}   x_h)   \cdots    (y_{\tau(i)}   x_h) =x'_\ell \,  y'_{\tau(1)}  \cdots  y'_{\tau(i)} , $$
so that the above sum $(*)$ is 
 an element of $C^S_{n-1} (X)$ as desired.
\end{proof}

The construction of 2-cycles in (i) is a generalization of   \cite{Zab}.
The following are immediate corollaries of Theorem~\ref{thm:main} for specific identities.

\begin{corollary}
Let $X$ be a type $n$ quandle, with the identity  $S$: $x *^n y=x$ for all $x, y \in X$ for a fixed $n >1$. 
Then the following properties hold.
\begin{itemize}
\item[{\rm (i)}] The $2$-chain $L_S = \sum_{i=0}^{k-1} ( x *^i  y, y )$ is a $2$-cycle, $L_S \in Z_2^R(X)$.

\item[{\rm (ii)}] For an abelian group $A$ and a $2$-cocycle $\phi$,  $E(X,A, \phi)$ satisfies $S$ if and only if 
$\phi( L_S)=0$
for any element $x, y \in X$. 

\item[{\rm (iii)}] 
Let $C^S_n(X)$ be the subgroup of $C_n(X)$
generated by 
$$
\bigcup_{j=0}^{n-1} \{ \sum_{i=0}^{k-1} 
( x_1*^i  y, \ldots,  x_j *^i y , y, x_{j+2},  \ldots, x_n)  \} .
$$
Then the sequence of  subgroups  $\{ C^S_n (X), \partial_n \}$ is a subcomplex.  
\end{itemize}
\end{corollary}

The following is motivated from Burnside relations discussed in  \cite{NP}. 
We consider the identity $xw=x$ for 
$$w= \underbrace{y_1 *y_2 * y_1 * y_2 * \cdots *  y_1 * y_2}_{k\  {\rm repetitions}} . $$
For simplicity denote $w$ by $Y^k$ where $Y$ denotes $y_1 y_2$ and the exponent $k$ represents the number of repetitions
of $y_1 y_2$ (but each $y_1 y_2$ is not parenthesized).

\begin{corollary}
Let $X$ be a  quandle that  satisfies the identity $S$: $xw=x$ for $w=Y^k$ as above,
 for all $x, y_1, y_2 \in X$  for a fixed $k >1$. 
Then the following properties hold.
\begin{itemize}
\item[{\rm (i)}] The $2$-chain 
$L_S = \sum_{i=0}^{k-1} \ [ \  ( x Y^i, y_1) + (xY^i y_1, y_2) \ ] $ is a $2$-cycle, $L_S \in Z_2^R(X)$.

\item[{\rm (ii)}] For an abelian group $A$ and a $2$-cocycle $\phi$,  $E(X,A, \phi)$ satisfies $S$ if and only if 
$\phi( L_S)=0$
for any element $x, y_1, y_2 \in X$. 

\item[{\rm (iii)}] 
Let $C^S_n(X)$ be the subgroup of $C_n(X)$
generated by 
\begin{eqnarray*}
\lefteqn{
\bigcup_{j=0}^{n-1} \ \{ \sum_{i=0}^{k-1} \ 
\ [ \  ( x_1 Y^i, \cdots, x_j Y^i, y_1, x_{j+2}, \ldots, x_n ) 
}\\
& & 
\qquad + ( x_1 Y^i y_1, \cdots, x_j Y^i y_1, y_2, x_{j+2}, \ldots, x_n ) \ ] 
\  \} .
\end{eqnarray*}
Then the sequence of  subgroups  $\{ C^S_n (X), \partial_n \}$ is a subcomplex.  
\end{itemize}
\end{corollary}

\begin{remark}
{\rm 
A quandle  $X$ is {\it medial}, or {\it abelian}, if the identity $S$: 
$$(x*y)*(u*v)=(x*u)*(y*v)$$ holds
for any $x,y,u,v \in X$. This property is well known, see \cite{JPSZ} for 
some discussions.
We note that this is not in the form of inner identity, but point out that procedures analogous to those
in the proof of Theorem~\ref{thm:main}  (i) and (ii) still apply. 
Let 
$L_S= [\, (x,y) + (x*y, u*v)\,  ] - [\, (x, u) + (x*u, y*v) \, ] $.
It holds that if $X$ is medial, then  (i) $L_S \in Z_2^R(X)$ for any $x,y,u,v \in X$, and (ii)
for a $2$-cocycle $\phi$ with a coefficient abelian group $A$, $E(X,A,\phi)$ is medial if $\phi(L(S))=0$. 

Proof is direct computations. For (i), we compute
\begin{eqnarray*}
\partial (L_S) &=& [\, (x) - (x*y)  + (x*y) - ((x*y)*(u*v) )\,  ] \\
& & - [\, (x) - (x*u) + (x*u) - ((x*u)*(y*v)) \, ] \quad = \quad 0 ,
\end{eqnarray*}
and for (ii), we compute
\begin{eqnarray*}
\lefteqn{ (\ (x,a) * (y,b) \ ) * (\ (u, c) * (v, d) \ ) } \\
 &=& (x*y, a+ \phi ( x, y) ) * (u*v, c + \phi( u, v) ) \\
 &=& ( ( x*y ) * (u*v) , a+ \phi ( x, y) + \phi ( x*y, u*v)  ),  \\
\lefteqn{ (\ (x,a) * (u,c) \ ) * (\ (y,b) * (v, d) \ ) } \\
 &=& (x*y, a+ \phi ( x, u) ) * (y*v, b + \phi( y, v) ) \\
 &=& ( ( x*u ) * (y*v) , a+ \phi ( x, u) + \phi ( x*u, y*v)  ).
 \end{eqnarray*}
}
\end{remark}

\section{Inner identities}\label{sec:id}

In this section we examine  
specific inner identities, as well as  quandles that satisfy these identities.
First we present  the number of type $n$ Rig quandles for possible values of $n$. 
The following list of vectors $[k, m]$ represent
that there are $m$ Rig quandles (all connected quandles of order $<48$) of type $k$. 

$$
\begin{array}{llllllll}
 [ 2, 117 ] & [ 3, 38 ] & [ 4, 90 ] & [ 5, 16 ] & [ 6, 117 ] & [ 7, 15 ] &   [ 8, 38 ] & [ 9, 13 ]  \\ 
 { } [ 10, 31 ] & [ 11, 10 ] & [ 12, 52 ] & [ 13, 4 ] & [ 14 , 19 ] & [ 15, 14 ] & [ 16, 9 ] & [ 18, 27 ] \\
{ }   [ 20, 19 ] & [ 21, 14 ] & [ 22 , 11 ] & [ 23, 22 ] & [ 24, 9 ] & [ 26, 5 ] & [ 28, 17 ] & [ 30, 15 ]\\
{ }   [ 31 , 6 ] & [ 36, 12 ] & [ 40, 16 ] & [ 42, 12 ] & [ 46, 22 ] & & &
\end{array}
$$

\begin{remark}
{\rm
Let $X$ be a quandle. 
The subgroup of ${\rm Sym}(X)$ generated by right transformations $\{ R_x \ | \ x \in X\}$ is called 
the {\it inner automorphism group} and denoted by ${\rm Inn}(X)$. 
For a quandle $X$, the map ${\rm inn}: X \rightarrow {\rm Inn}(X)$ defined by 
${\rm inn}(a)=R_a$ for $a \in X$ is called the {\it inner representation}.
Computer calculations show that for over 3000 non-faithful quandles $X$ (mostly generalized Alexander quandles)
it holds that ${\rm type}(X)={\rm type}({\rm inn}(X))$.
For the great majority of these quandles,
${\rm inn}: X \rightarrow {\rm inn}(X)$  are abelian extensions.
In \cite{CSV}, it was conjectured that if a quandle $X$ is a kei, then 
any abelian extension of $X$ is a kei. 
Thus we make the following conjectures.
}
\end{remark}

\begin{conjecture}
{\rm 
(1) 
If $\alpha: E \rightarrow X$ is a connected abelian extension then ${\rm  type}(E)={\rm  type}(X)$.

\noindent 
(2) 
If $Q$ is connected then 
${\rm type}(Q) = {\rm type}({\rm inn}(Q))$.
}
\end{conjecture}

Let $X$ be a quandle. Let $xw=x$ be an inner identity, where $w$ is a word in  the alphabet $\Lambda$.
Let $|w|$ denote the length of $w$, that is, the number of letters in $w$.
We note that if the length of $w$ is $k$, then the quandle is of type at most $k$, since the identity must hold for any values 
of variables. We observe the following.

\begin{lemma}\label{lem:triv}
Let $w$  be a word in  the alphabet $\Lambda$ such that some
letter of the alphabet, say, $a$, appears only once in $w$. Then a quandle
satisfying $xw = x$ is  trivial.
\end{lemma}

\begin{proof}
Convert this identity to a product of $R_c$, $c \in X$. 
The identity is equivalent to that his product
equals 1,  so we can solve it for $R_a$. Thus $R_a$ is a product of $R_c^{\pm 1}$, $c \neq a$.
This identity must hold for any $a, c \in X$. By fixing the values of $c$ and choosing different values of $a$, 
 we obtain $R_a = R_{a'}$ for
all $a, a' \in X$.  Since $R_a(a)=a*a=a$, we obtain $R_{a'}(a)=a*a'=a$ for all $a, a' \in X$. 
\end{proof}

\begin{corollary}
 If $Q$ is a non-trivial quandle satisfying an identity $xw =x$ where $w$ is
a word in alphabet $\Lambda$, then each letter appearing in $w$  must appear
at least twice.
\end{corollary}

\begin{lemma}\label{lem:type}
If a quandle $X$ satisfies $xw=x$ where $w$ has two letters one of which  appears consecutively $k$ times,
then 
${\rm type}(X) \leq {\rm gcd}\{ k, |w| - k \}$.
\end{lemma}

\begin{proof}
\begin{sloppypar}
Under the assumption $w$ is written as $w=a^h b^k a^{|w|-h-k}$ for $h \geq 0$, where the exponents represents the number of  repetitions.
Then the identity $xw=x$ is converted to the identity 
 $(R_a)^{|w| - h - k}  (R_b)^k (R_a)^h=1$ in ${\rm Inn}(X)$. 
 \end{sloppypar}
Hence $(R_b)^k = (R_a^{-1})^{|w|-k}$, so that 
$a*^kb=(R_b)^k (a)=(R_a^{-1})^{|w|-k}(a)=a$, 
and $$b*^{|w|-k} a=(R_a)^{|w|-k}(b)=(R_b^{-1})^k (b)= b$$  for all $a, b \in X$, 
as desired.
\end{proof}

We examine identities of small lengths.

\bigskip

\noindent
{\it Length 1}. If the identity $xa=x$ holds in a quandle $X$, then $X$ is trivial by Lemma~\ref{lem:triv}.

\bigskip

\noindent
{\it Length 2}. The identity $xaa=x$ holds in a quandle $X$ if and only if $X$ is a kei by Lemma~\ref{lem:type}.
If the identity $xab=x$ holds, then the quandle is trivial by Lemma~\ref{lem:triv}.

\bigskip

\noindent
{\it Length 3}. The identity $xaaa=x$ holds in a quandle $X$ if and only if $X$ is of type 3.
Other cases are trivial quandles by Lemma~\ref{lem:triv}.

\bigskip

\noindent
{\it Length 4}. Excluding trivial quandles from Lemma~\ref{lem:triv}, we have the following cases.
(1) $xaabb=x$,
(2) $xabba=x$,
(3) $xabab=x$.
Lemma~\ref{lem:type} implies that (1) or (2) holds if and only if the quandle is  a kei.

Computer calculation shows that among 790 Rig quandles of order less than 48, the following 
quandles satisfy the identity $xabab=x$, none of which is a kei.
$$
\begin{array}{lllllll}
Q(5,2) & Q(5,3) & Q(9,3) & Q(13,4) & Q(13,7) & Q(17,3) & Q(17,12) \\
 Q(25,4) &  Q(25,5) & Q(25,6) & Q(25,7)& Q(25, 8 ) & Q(29, 11) & Q(29, 16) \\
 Q(37, 45) & Q(37, 5) & Q(41,2) & Q(41, 3) &  Q(45, 36) & Q(45, 37) & 
\end{array}
$$

\noindent
{\it Length 5}. Excluding trivial quandles from Lemma~\ref{lem:triv}, there are 10 identities $xw=x$ (all must have  two distinct  letters in $w$),
$$w=aaabb, \ aabab, \  aabba, \ abaab, \ ababa, \ abbaa, \ aabbb, \ ababb, \ abbab, \ abbba. $$
From Lemma~\ref{lem:type},  the identities $aaabb$, $aabba$, $abbaa$, $aabbb$,  $abbba$ imply trivial quandle.
Computer calculation shows that none of the remaining is satisfied by any of 790 Rig quandles.
We conjecture that no connected quandle satisfies the remaining 5  identities of length $5$.

\bigskip

\noindent
{\it Length 6}. Words $w$ of length 6 with 2 letters, excluding those implying trivial and type 2, 3 quandles
from Lemmas~\ref{lem:triv}, \ref{lem:type}, consist of the following list.

The words $w$ such that identities $xw=x$ are not satisfied by any of the Rig quandles are:
\begin{eqnarray*}
w &=& aaabab, \ %17
aababa, \ %4 
aabbab, \ %13
aababb, \ %15
abaaab, \ %21
abaabb, \\ %22
& & ababbb, \ %19
abbaba, \ %6
abbaab, \ %24
ababaa, \ %9
ababba, \ %10
abbbab.  %25
\end{eqnarray*}
Thus we conjecture that no connected quandle satisfies these.

The following words are satisfied by the same 202 Rig quandles, which contain all 117 Rig keis: 
$aabaab$,  %14
$abaaba$,  %11
$abbabb$.  %23

The word 
$ababab$ is satisfied by 55 Rig quandles, 4 of which are keis. %20

\bigskip

\noindent
{\it Length 7}.
Except for those words $w$ that give trivial quandles from Lemmas~\ref{lem:triv}, \ref{lem:type}
and type 7 quandles (15 of them among 790 Rig quandles), there are only two Rig quandles that satisfy the identity $xw=x$ with 2 letter words $w$ of  length 7, and they are: 

$Q(8,2)=\mathbb{Z}_2[t]/(t^3 + t^2 +1)$ satisfies identities with 
$$w=aababba, %10
abbbaba, %18
ababbaa, %23
aabbbab, %28
aaababb, %38
abaabbb, %49
abbaaab. $$%52

$Q(8,3)=\mathbb{Z}_2[t]/(t^3 + t +1)$ satisfies identities with 
$$w=aabbaba, %6
abbabaa, %15
ababbba, %22
aababbb, %34
aaabbab, %40
abaaabb, %46
abbbaab. $$%55

Finally we observe the following.

\begin{proposition}\label{prop:many}
For any $m, n \in \mathbb{Z}$ such that $m>0$ and   $n>1$,   there exist infinitely many  connected quandles 
that satisfy the identity $xw=x$ for 
$$w= \underbrace{y_1 * \cdots * y_m * y_1 \cdots * y_m *\cdots y_1 \cdots * y_m}_{n\  {\rm repetitions}}. $$
\end{proposition}

\begin{proof}
We consider Alexander quandles $(X, t)$, where $t$ is an automorphism of an abelian group $X$,
with $x*y=tx + (1-t)y$. 
Inductively one computes
$$x*y_1 * \cdots * y_k = t^k x + (1-t) ( t^{k-1} y_1 + t^{k-2} y_2 + \cdots + t y _{k-1} + y_k ) . $$
Setting $k=mn$ and $y_{hm+j}=y_j$ for $h=0, 1, \ldots, n-1$, $j=1, \ldots, m$, 
we obtain 
\begin{eqnarray*}
\lefteqn{ x* \underbrace{y_1 * \cdots * y_m * y_1 \cdots * y_m *\cdots y_1 \cdots * y_m}_{n\  {\rm repetitions}} } \\
&=&  t^{mn} x + (1-t)( t^{mn-m} + t^{mn-2m} + \cdots + t^m + 1  ) \\
 & & \times  ( t^{m-1} y_1 + y^{m-2} y_2 + \cdots + t y _{m-1} + y_m ) . 
\end{eqnarray*}
Let $g_{m,n}(t)=t^{mn-m} + t^{mn-2m} + \cdots + t^m + 1 =(t^{mn} -1)/(t^m -1)$ and 
$X=\mathbb{Z}_p [t]/( g_{m,n}(t)  )$. 
For all primes $p>n$, $g_{m,n} (1) \neq 0$, so that   $1-t$ is invertible in $X$ and hence $X$ is connected, and $t^{mn}=1$ in $X$.
Hence  there is 
an infinite family of connected quandles that satisfy $xw=x$.
\end{proof}

Note that in Proposition~\ref{prop:many}, $n=1$ is not possible for any choice of a word $ y_1 \cdots  y_m$ by Lemma~\ref{lem:triv}.

\begin{remark}
{\rm

For a group $G$, the least possible integer $n$ such that $x^n = 1$ 
for all $x \in G$ is called the {\it exponent} of $G$. 
We note that if the exponent of ${\rm  Inn}(X)$ is $n$, then 
$X$ satisfies the inner identity  $xw^n = x$ for any word $w$ of length $n$,
 where $\omega^n$ denotes $\omega$ repeated 
$n$ times with no parentheses. 
We computed the exponent of ${\rm  Inn}(X)$ by \textsf{GAP} \  for all 
790 Rig quandles, and obtained the following data.
The notation $[e, n]$ below indicates that there are $n$ Rig quandles $X$ such that  the exponent of ${\rm Inn}(X)$ is $e$. 
The pairs are listed in order
of increasing exponent.
$$
\begin{array}{llllllll}
 [ 6, 11 ] & [ 10, 4 ] & [ 12, 59 ] & [ 14, 3 ] & [ 15, 1 ] & [ 18, 47 ] & 
  [ 20, 15 ] \\
{ }   [ 21, 2 ] &
{ }    [ 22, 1 ] & [ 24, 38 ] & [ 26, 1 ] & [ 30, 22 ] & 
  [ 34, 1 ] & [ 36, 31 ] \\
  { }  [ 38, 1 ] & [ 39, 6 ] &
{ }    [ 40, 6 ] & [ 42, 22 ] & 
  [ 46, 1 ] & [ 48, 4 ] & [ 50, 5 ] \\
  { }  [ 52, 2 ] & [ 54, 9 ] & [ 55, 4 ] & %
{ }   [ 57, 2 ] & [ 58, 1 ] & [ 60, 44 ] & [ 62, 7 ] \\
{ }  [ 66, 4 ] & [ 68, 2 ] & 
  [ 70, 3 ] & [ 72, 13 ]  & 
{ }    [ 74, 1 ] & [ 78, 13 ] & [ 82, 1 ] \\
{ }  [ 84, 24 ] & 
  [ 86, 1 ] & [ 90, 9 ] & [ 93, 2 ] & [ 94, 1 ]  & 
 { }   [ 100, 10 ] & [ 110, 4 ] \\ 
 { }  [ 111, 2 ] & [ 114, 2 ] & [ 116, 2 ] & [ 120, 27 ] & [ 129, 2 ] & [ 136, 4 ]  & 
{ }   [ 140, 6 ] \\
{ }  [ 148, 2 ] & [ 155, 4 ] & [ 156, 10 ] & [ 164, 2 ] & [ 168, 4 ] & 
  [ 171, 6 ] & [ 180, 12 ] \\
 { }   [ 186, 2 ] & [ 203, 6 ] & [ 205, 4 ] & [ 210, 4 ] & 
  [ 222, 2 ] & [ 240, 3 ] & [ 253, 10 ] \\
  { }  [ 258, 2 ] & 
 { }  [ 272, 8 ] & [ 301, 6 ] & 
  [ 310, 4 ] & [ 328, 4 ] & [ 330, 16 ] & [ 333, 6 ] \\
  { }  [ 342, 6 ] & [ 360, 1 ]  & 
 { }  [ 406, 6 ] & [ 410, 4 ] & [ 420, 10 ] & [ 444, 4 ] & [ 465, 8 ] \\
 { }  [ 506, 10 ] & 
  [ 602, 6 ] & [ 666, 6 ]  & 
 { }  [ 812, 12 ] & [ 820, 8 ] & [ 840, 3 ] & [ 903, 12 ] \\  
{ }   [ 930, 8 ] & [ 1081, 22 ] & [ 1332, 12 ] & [ 1640, 16 ] & 
{ }   [ 1806, 12 ] & 
  [ 2162, 22 ] & [ 2520, 2 ] 
\end{array}
$$

}
\end{remark}

\subsection*{Acknowledgements}
We are grateful to  Jozef Przytycki for valuable comments. 
M.S.  was partially supported by
the
NIH  1R01GM109459. 
The content of
this paper is solely the responsibility of the authors and does not necessarily
represent the official views of  NIH.

\end{document}